\documentclass[british,a4paper]{scrartcl}
\usepackage[T1]{fontenc}
\usepackage[latin9]{inputenc}
\usepackage[a4paper]{geometry}
\usepackage{babel}
\usepackage{prettyref}
\usepackage{mathtools}
\usepackage{amsmath}
\usepackage{amsthm}
\usepackage{amssymb}
\usepackage[unicode=true,pdfusetitle,
 bookmarks=true,bookmarksnumbered=false,bookmarksopen=false,
 breaklinks=false,pdfborder={0 0 1},backref=false,colorlinks=false]
 {hyperref}

\makeatletter
\theoremstyle{plain}
\newtheorem{thm}{\protect\theoremname}[section]
\theoremstyle{plain}
\newtheorem{lem}[thm]{\protect\lemmaname}
\theoremstyle{definition}
\newtheorem*{defn*}{\protect\definitionname}
\theoremstyle{remark}
\newtheorem{rem}[thm]{\protect\remarkname}
\theoremstyle{plain}
\newtheorem{prop}[thm]{\protect\propositionname}
\theoremstyle{plain}
\newtheorem{cor}[thm]{\protect\corollaryname}

\@ifundefined{date}{}{\date{}}
\usepackage{prettyref}

\usepackage{enumerate}
\allowdisplaybreaks

\newcommand*{\e}{\mathrm{e}}
\renewcommand*{\i}{\mathrm{i}}

\newcommand{\R}{\mathbb{R}}
\newcommand{\N}{\mathbb{N}}

\newcommand{\dom}{\operatorname{dom}}
\newcommand{\ran}{\operatorname{ran}}

\renewcommand{\d}{\,\mathrm{d}}
\renewcommand{\Re}{\operatorname{Re}}

\newcommand{\grad}{\operatorname{grad}}

\newcommand{\curl}{\operatorname{curl}}
\newcommand{\dive}{\operatorname{div}}

\newrefformat{subsec}{Subsection \ref{#1}}
\newrefformat{prob}{Problem \ref{#1}}
\newrefformat{prop}{Proposition \ref{#1}}
\newrefformat{lem}{Lemma \ref{#1}}
\newrefformat{thm}{Theorem \ref{#1}}
\newrefformat{cor}{Corollary \ref{#1}}
\newrefformat{rem}{Remark \ref{#1}}
\newrefformat{exa}{Example \ref{#1}}
\newrefformat{sub}{Subsection \ref{#1}}
\newrefformat{eq}{(\ref{#1})}

\theoremstyle{definition}

\newrefformat{hyp}{Hypotheses \ref{#1}}

\makeatother

\providecommand{\corollaryname}{Corollary}
\providecommand{\definitionname}{Definition}
\providecommand{\lemmaname}{Lemma}
\providecommand{\propositionname}{Proposition}
\providecommand{\remarkname}{Remark}
\providecommand{\theoremname}{Theorem}

\begin{document}
\title{M-Accretive Realisations of Skew-Symmetric Operators}
\author{Rainer Picard\thanks{Institut für Analysis, TU Dresden, Dresden, Germany, rainer.picard@tu-dresden.de}
and Sascha Trostorff\thanks{Mathematisches Seminar, CAU Kiel, Kiel, Germany, trostorff@math.uni-kiel.de}}
\maketitle
\begin{abstract}
\textbf{Abstract. }We consider skew-symmetric operators $A_{0}$ on
a Hilbert space $H$ and characterise all (nonlinear) m-accretive
restrictions of $A\coloneqq-A_{0}^{\ast}$ in terms of the `deficiency
spaces' $\ker(1\pm A)$. The results are illustrated by several examples
and applied to a partial differential equation with an impedance type
boundary condition. 
\end{abstract}

\section{Introduction}

In mathematical physics, the realisation of certain -- frequently
differential -- operators play an essential role. For instance, in
quantum mechanics, selfadjoint realisations of symmetric operators
are of key interest while for well-posedness of partial differential
equations, different realisations like selfadjoint, skew-selfajoint
and accretive ones are a common subject of study. For all these cases,
different frameworks and techniques have been developed. 

The probably most prominent technique to characterise all selfadjoint
extensions of symmetric operators $S$ goes back to von Neumann, \cite{vonNeumann},
and involves the so-called deficiency spaces $\ker(\i\pm S^{\ast})$
and the Cayley transform associated with $S$ (see e.g. \cite[Chapter 4]{Weidmann1980},
\cite[Chapter 13.2]{Schmudgen2021}). Other techniques involving forms
(see e.g. \cite{SSVW2015}) or boundary triplets (\cite[Chapter 3]{Gorbachuk1991},
\cite[Chapter 14]{Schmudgen2021}) have been used to study selfadjoint,
skew-selfadjoint and accretive realisations. Note that the framework
of boundary triplets and the generalised concept of boundary relations
has been successfully used in spectral theory for linear operators
and more generally linear relations (see e.g. \cite{Derkach2012,Behrndt2020}
and the references therein). 

We emphasise that in all the references above, the realisations are
assumed to be linear. So far, the topic of nonlinear realisations
seems not to have been addressed extensively in the literature. We
are only aware of a paper by Posilicano \cite{Posilicano}, where
m-accretive nonlinear extensions of accretive symmetric operators
are considered and by one of the authors in \cite{Trostorff2014},
where nonlinear m-accretive realisations of skew-symmetric block operator
matrices are considered. We also mention \cite{Augner}, where m-accretive
realisations were used to study port-Hamiltonian systems with nonlinear
boundary feedback. 

In the present paper, we provide a characterisation of all possible
nonlinear m-accretive restrictions of an operator $A\coloneqq-A_{0}^{\ast}$,
where $A_{0}$ is a skew-symmetric operator on some Hilbert space.
The characterisation follows the spirit of the von Neumann approach
for symmetric operators mentioned above, by considering the spaces
$\ker(1\pm A)$ and contractions between them. 

The motivation to study such realisations lies in the study of evolution
equations, where m-accretivity (or equivalently m-dissipativity) is
a key assumption for differential operators to ensure well-posedness.
Indeed, in the language of $C_{0}$-semigroups, m-accretive operators
correspond to generators of contractive semigroups by the famous theorem
of Lumer-Phillips in the linear case (see e.g. \cite[Theorem 3.15]{Engel_Nagel})
and by the work of Komura in the nonlinear case, \cite{Komura} (also
note the extension to Banach spaces by Crandall and Pazy, \cite{Crandall}).
Moreover, m-accretive operators play a crucial role in the theory
of evolutionary equations or inclusions, see e.g. \cite{Picard2009},
\cite[Chapter 7]{Picard_McGhee}, \cite{Trostorff2020} and \cite{STW2022}. 

The article is structured as follows. In Section 2 we collect some
preliminary results on skew-symmetric operators and m-accretive relations.
These results are used to prove the main result in Section 3 (\prettyref{thm:main})
providing a characterisation of all m-accretive realisations of $A\coloneqq-A_{0}^{\ast}\supseteq A_{0}$
for some skew-symmetric operator $A_{0}.$ Section 4 is devoted to
two examples. Firstly, we discuss the derivative on an interval and
use our abstract characterisation result to describe all m-accretive
realisations of the derivative in terms of boundary conditions. Secondly,
we apply the characterisation result to block operator matrices, which
naturally occur in many equations of mathematical physics. The characterisation
is done in terms of abstract boundary data spaces, which turn out
to be equivalent to classical trace spaces in case of gradient, divergence
or rotation on smooth domains. We hereby provide another and simpler
proof of the main result of \cite{Trostorff2014}. The article is
concluded by a concrete example illustrating the applicability of
the abstract result to differential equations with impedance type
boundary conditions. 

Throughout, all Hilbert spaces are without loss of generality\footnote{Note that a complex Hilbert space $H$ with inner product $\langle\cdot,\cdot\rangle_{H}$
can also be viewed as a real Hilbert space by restricting scalar multipliers
to $\mathbb{R}$ and using $\langle\cdot,\cdot\rangle\coloneqq\Re\langle\cdot,\cdot\rangle$
as inner product without affecting the topology of $H$.} assumed to be real and their inner product is denoted by $\langle\cdot,\cdot\rangle$
with suitable indeces. 

\section{Preliminaries}

Throughout, let $H$ be a Hilbert space and $A_{0}:\dom(A_{0})\subseteq H\to H$
be a skew-symmetric closed linear operator and set $A\coloneqq-A_{0}^{\ast}\supseteq A_{0}.$
For a closed operator $B:\dom(B)\subseteq X\to Y$ between to Hilbert
spaces $X$ and $Y$ we equip $\dom(B)$ with the graph inner product.
In that way, $\dom(B)$ becomes a Hilbert space itself. Since $A_{0}\subseteq A$
we clearly have $\dom(A_{0})\subseteq\dom(A)$ and $\dom(A_{0})$
is a closed subspace of $\dom(A).$ We first compute the orthogonal
complement of this subspace.
\begin{lem}
\label{lem:decomp_domain}We have
\[
\dom(A)=\dom(A_{0})\oplus_{A}\ker(1-A^{2})=\dom(A_{0})\oplus_{A}\ker(1-A)\oplus_{A}\ker(1+A),
\]
where the orthogonal sum is taken with respect to the graph inner
product in $\dom(A)$.
\end{lem}

\begin{proof}
We start to compute the orthogonal complement of $\dom(A_{0})$ in
$\dom(A).$ We have
\begin{align*}
u\in\dom(A_{0})^{\bot_{\dom(A)}} & \Leftrightarrow\forall v\in\dom(A_{0}):\,\langle u,v\rangle_{\dom\left(A\right)}=0,\\
 & \Leftrightarrow\forall v\in\dom(A_{0}):\,\langle u,v\rangle_{H}+\langle Au,A_{0}v\rangle_{H}=0,\\
 & \Leftrightarrow Au\in\dom(A_{0}^{\ast}):\:A_{0}^{\ast}Au=-u,\\
 & \Leftrightarrow u\in\dom(A^{2}):\,u-A^{2}u=0.
\end{align*}
The latter shows $\dom(A_{0})^{\bot_{\dom\left(A\right)}}=\ker(1-A^{2})$
and hence, the first orthogonal decomposition follows. For the second
equality, we note that $\ker(1\pm A)\subseteq\ker(1-A^{2}):$ Indeed,
for $u\in\ker(1\pm A)$ we have $u\in\dom(A)$ and $u\pm Au=0.$ The
latter implies $Au=\mp u\in\dom(A)$ and $A^{2}u=\mp Au=u.$ Moreover,
the subspaces $\ker(1+A)$ and $\ker(1-A)$ are orthogonal in $\dom(A).$
Indeed, for $u\in\ker(1+A),\,v\in\ker(1-A)$ we have 
\[
\langle u,v\rangle_{\dom(A)}=\langle u,v\rangle_{H}+\langle Au,Av\rangle_{H}=\langle u,v\rangle_{H}+\langle u,-v\rangle_{H}=0.
\]
It remains to prove that each element $u\in\ker(1-A^{2})$ can be
decomposed in $u=v+w$ for some $v\in\ker(1-A),\,w\in\ker(1+A).$
We set $v\coloneqq\frac{1}{2}(1+A)u\in\dom(A)$ and $w\coloneqq\frac{1}{2}(1-A)u\in\dom(A).$
Then clearly $\left(1-A\right)v=\frac{1}{2}(1-A^{2})u=0$ and $(1+A)w=\frac{1}{2}(1-A^{2})u=0$
and $v+w=u$, which shows the desired decomposition. 
\end{proof}
\begin{defn*}
We define the orthogonal projections according to the decomposition
in \prettyref{lem:decomp_domain} by 
\begin{align*}
\pi_{0} & \colon\dom(A)\to\dom(A_{0}),\\
\pi_{1} & \colon\dom(A)\to\ker(1-A),\\
\pi_{-1} & \colon\dom(A)\to\ker(1+A),
\end{align*}
respectively. Similarly, we denote the canonical embeddings by $\iota_{0},\iota_{1}$
and $\iota_{-1}$, respectively.
\end{defn*}
\begin{rem}
Note that $\pi_{0}=\iota_{0}^{\ast},\,\pi_{1}=\iota_{1}^{\ast}$ and
$\pi_{-1}=\iota_{-1}^{\ast}$, where the adjoints are computed with
respect to the graph inner product of $A$, i.e. in the Hilbert space
$\dom(A).$ Moreover, for $u\in\dom(A)$ we have 
\[
Au=A\pi_{0}u+A\pi_{1}u+A\pi_{-1}u=A_{0}\pi_{0}u+\pi_{1}u-\pi_{-1}u.
\]
\end{rem}

\begin{lem}
\label{lem:inner_product_A}For $u\in\dom(A)$ we have 
\[
\langle Au,u\rangle_{H}=\|\pi_{1}u\|_{H}^{2}-\|\pi_{-1}u\|_{H}^{2}.
\]
\end{lem}

\begin{proof}
We compute 
\begin{align*}
\langle Au,u\rangle_{H} & =\langle A_{0}\pi_{0}u+\pi_{1}u-\pi_{-1}u,\pi_{0}u+\pi_{1}u+\pi_{-1}u\rangle_{H}\\
 & =\|\pi_{1}u\|_{H}^{2}-\|\pi_{-1}u\|_{H}^{2}+\langle A_{0}\pi_{0}u,\pi_{1}u+\pi_{-1}u\rangle_{H}+\\
 & \quad+\langle\pi_{1}u,\pi_{0}u\rangle_{H}-\langle\pi_{-1}u,\pi_{0}u\rangle_{H},
\end{align*}
where we have used $\langle A_{0}\pi_{0}u,\pi_{0}u\rangle_{H}=0$
due to the skew-symmetry of $A_{0}.$ Moreover, we compute 
\begin{align*}
\langle A_{0}\pi_{0}u,\pi_{1}u+\pi_{-1}u\rangle_{H} & =-\langle\pi_{0}u,A\pi_{1}u+A\pi_{-1}u\rangle_{H}\\
 & =-\langle\pi_{0}u,\pi_{1}u\rangle_{H}+\langle\pi_{0}u,\pi_{-1}u\rangle_{H}
\end{align*}
and thus, the assertion follows.
\end{proof}
The kernels $\ker(1-A)$ and $\ker(1+A)$ are also closed subspaces
of the Hilbert space $H$ and thus, there exist orthogonal projections
\begin{align*}
P_{\ker(1-A)} & :H\to\ker(1-A),\\
P_{\ker(1+A)} & :H\to\ker(1+A).
\end{align*}
The next lemma relates these projections with the projections $\pi_{1}$
and $\pi_{-1}.$ 
\begin{lem}
\label{lem:projections} Let $u\in\dom(A).$ Then 
\begin{align*}
\pi_{1}u & =\frac{1}{2}P_{\ker(1-A)}(1+A)u,\\
\pi_{-1}u & =\frac{1}{2}P_{\ker(1+A)}(1-A)u.
\end{align*}
\end{lem}

\begin{proof}
The reasoning for the second equality being analogous, we only show
the first. It suffices to check that $\pi_{1}u-\frac{1}{2}(1+A)u$
is orthogonal to $\ker(1-A)$ with respect to $H$. For this, let
$v\in\ker(1-A)$ and compute 
\begin{align*}
\langle\pi_{1}u-\frac{1}{2}(1+A)u,v\rangle_{H} & =\frac{1}{2}\langle\pi_{1}u-u,v\rangle_{H}+\frac{1}{2}\langle\pi_{1}u-Au,v\rangle_{H}\\
 & =\frac{1}{2}\langle\pi_{1}u-u,v\rangle_{H}+\frac{1}{2}\langle A\pi_{1}u-Au,Av\rangle_{H}\\
 & =\frac{1}{2}\langle\pi_{1}u-u,v\rangle_{\dom(A)}=0,
\end{align*}
where we have used that $Av=v$ and $A\pi_{1}u=\pi_{1}u$.
\end{proof}
We conclude this section by stating some useful facts on $m$-accretive
operators. For later use we recall the definition in the framework
of binary relations:
\begin{defn*}
Let $B\subseteq H\times H$. Then $B$ is called \emph{accretive},
if 
\[
\forall(u,v),(x,y)\in B:\,\langle y-x,u-v\rangle_{H}\geq0.
\]
Moreover, $B$ is $m$\emph{-accretive}, if $B$ is accretive and
$1+B$ is onto; i.e., for each $f\in H$ there exists $(u,v)\in B$
such that 
\[
u+v=f.
\]
\end{defn*}
\begin{rem}
For an accretive operator $B$, we have that $1+B$ is onto if and
only if $\lambda+B$ is onto for some $\lambda>0.$ Moreover, the
celebrated Theorem of Minty, see \cite{Minty1960} or \cite[Theorem 17.1.7]{STW2022},
states that an accretive operator $B$ is $m$-accretive if and only
if it is maximal in the set of accretive relations. In particular,
if $B$ is $m$-accretive and $C$ is accretive with $B\subseteq C,$
then $B=C.$ 
\end{rem}

For later purposes, we state the following simple observation.
\begin{prop}
\label{prop:accretive_dense_range}Let $B\subseteq H\times H$ be
accretive. Then the following statements are equivalent:

\begin{enumerate}[(i)]

\item $B$ is $m$-accretive,

\item $B$ is closed and $\ran(1+B)$ is dense in $H$.

\end{enumerate}
\end{prop}

\begin{proof}
(i) $\Rightarrow$ (ii): If $B$ is $m$-accretive, then $\ran(1+B)=H.$
For showing the closedness of $B,$ let $(x_{n},y_{n})_{n}$ be a
sequence in $B$ with $x_{n}\to x$ and $y_{n}\to y$ in $H$ for
some $x,y\in H$. By assumption, we find $(u,v)\in B$ such that 
\[
u+v=x+y.
\]
Then we estimate for each $n\in\N$
\begin{align*}
\|x_{n}-u\|_{H}^{2} & =\langle x_{n}-u,x_{n}+y_{n}-(u+v)\rangle-\langle x_{n}-u,y_{n}-v\rangle_{H}\\
 & \leq\langle x_{n}-u,x_{n}+y_{n}-(x+y)\rangle_{H},
\end{align*}
where we have used the accretivity of $B$. Letting $n\to\infty$
in the latter inequality, we infer 
\[
\|x-u\|_{H}^{2}\leq0
\]
and hence, $x=u.$ The latter implies $y=v$ and thus, $(x,y)=(u,v)\in B,$
which shows that $B$ is closed.\\
(ii) $\Rightarrow$ (i): Let $f\in H.$ By assumption we find a sequence
$(f_{n})_{n}$ in $\ran(1+B)$ such that $f_{n}\to f.$ Moreover,
we find $(u_{n},v_{n})\in B$ with $u_{n}+v_{n}=f_{n}$ for each $n\in\N.$
For $n,m\in\N$ we estimate 
\begin{align*}
\|u_{n}-u_{m}\|_{H}^{2} & =\langle u_{n}-u_{m},f_{n}-f_{m}\rangle_{H}-\langle u_{n}-u_{m},v_{n}-v_{m}\rangle_{H}\\
 & \leq\langle u_{n}-u_{m},f_{n}-f_{m}\rangle_{H}\\
 & \leq\|u_{n}-u_{m}\|_{H}\|f_{n}-f_{m}\|_{H}
\end{align*}
and hence, $\|u_{n}-u_{m}\|_{H}\leq\|f_{n}-f_{m}\|_{H}\to0$ as $n,m\to\infty.$
Thus, we find $u\in H$ with $u_{n}\to u$ as $n\to\infty.$ The latter
implies 
\[
v_{n}=f_{n}-u_{n}\to f-u
\]
and thus, as $B$ is closed, we infer $(u,f-u)\in B,$ which yields,
$f\in\ran(1+B)$ and thus, $B$ is $m$-accretive.
\end{proof}

\section{A Characterisation of $m$-Accretive Restrictions}

As in the previous section, let $A_{0}:\dom(A_{0})\subseteq H\to H$
be closed and skew-symmetric on a Hilbert space $H$ and set $A\coloneqq-A_{0}^{\ast}.$
We recall the decomposition from \prettyref{lem:decomp_domain}
\[
\dom(A)=\dom(A_{0})\oplus_{A}\ker(1-A)\oplus_{A}\ker(1+A).
\]
We aim to prove the following theorem, which provides a characterisation
of all $m$-accretive restrictions of $A$.
\begin{thm}
\label{thm:main}Let $B\subseteq A.$ Then $B$ is $m$-accretive
if and only if there exists a mapping 
\[
h\colon\ker(1-A)\to\ker(1+A)
\]
with
\[
\forall x,y\in\ker(1-A):\,\|h(x)-h(y)\|_{H}\leq\|x-y\|_{H}
\]
such that 
\[
\dom(B)=\{u\in\dom(A)\,;\,h(\pi_{1}u)=\pi_{-1}u\}.
\]
\end{thm}

We start to prove the following partial result of this characterisation.
\begin{prop}
\label{prop:sufficient}Let 
\[
h\colon\ker(1-A)\to\ker(1+A)
\]
with
\[
\forall x,y\in\ker(1-A):\,\|h(x)-h(y)\|_{H}\leq\|x-y\|_{H}
\]
and define 
\begin{align*}
B\colon\{u\in\dom(A)\,;\,h(\pi_{1}u)=\pi_{-1}u\}\subseteq H & \to H,\\
u & \mapsto Au.
\end{align*}
Then $B$ is $m$-accretive.
\end{prop}

\begin{proof}
We first show that $B$ is accretive. For doing so, let $u,v\in\dom(B)$
and we compute, using \prettyref{lem:inner_product_A}, 
\begin{align*}
\langle B\left(u\right)-B\left(v\right),u-v\rangle_{H} & =\langle A(u-v),u-v\rangle_{H}\\
 & =\|\pi_{1}(u-v)\|_{H}^{2}-\|\pi_{-1}(u-v)\|_{H}^{2}\\
 & =\|\pi_{1}u-\pi_{1}v\|_{H}^{2}-\|h(\pi_{1}u)-h(\pi_{1}v)\|_{H}^{2}\\
 & \geq0.
\end{align*}
Next, we prove that $B$ is closed. For doing so, let $(u_{n})_{n}$
in $\dom(B)$ such that $u_{n}\to u$ and $B\left(u_{n}\right)\to v$
for some $u,v\in H.$ Since $B\subseteq A,$ we infer that $u\in\dom(A)$,
$v=Au$ and $u_{n}\to u$ in $\dom\left(A\right).$ Consequently,
$\pi_{1}(u_{n})\to\pi_{1}(u)$ and $\pi_{-1}(u_{n})\to\pi_{-1}(u)$
in $\dom\left(A\right)$ and hence, also in $H$. Thus, 
\[
h(\pi_{1}u)=\lim_{n\to\infty}h(\pi_{1}(u_{n}))=\lim_{n\to\infty}\pi_{-1}(u_{n})=\pi_{-1}u,
\]
and hence, $u\in\dom(B)$ with $v=Au=B\left(u\right)$, which shows
the closedness of $B$. In order to prove the $m$-accretivity of
$B$, it suffices to show that $\ran(1+B)$ is dense in $H$ by \prettyref{prop:accretive_dense_range}.
For doing so, let $f\in H$ and define $v\coloneqq P_{\ker(1-A)}f\in\ker(1-A).$
Then $f-v\in\ker(1-A)^{\bot_{H}}=\overline{\ran}(1+A_{0})$ and thus,
we find a sequence $(w_{n})_{n}$ in $\dom(A_{0})$ with 
\[
w_{n}+A_{0}w_{n}\to f-v\quad(n\to\infty).
\]
Now, define 
\[
u_{n}\coloneqq w_{n}+\tfrac{1}{2}v+h(\tfrac{1}{2}v).
\]
Then $u_{n}\in\dom(A)$ with $\pi_{1}u_{n}=\frac{1}{2}v,$ $\pi_{-1}u_{n}=h(\frac{1}{2}v)$
and thus, $u_{n}\in\dom(B)$ for each $n\in\N.$ Moreover, 
\begin{align*}
u_{n}+B(u_{n}) & =w_{n}+\tfrac{1}{2}v+h(\tfrac{1}{2}v)+A_{0}w_{n}+\tfrac{1}{2}v-h(\tfrac{1}{2}v)\\
 & =w_{n}+A_{0}w_{n}+v\to f\quad(n\to\infty)
\end{align*}
and hence, $f\in\overline{\ran}(1+B),$ proving the density of $\ran(1+B).$ 
\end{proof}
With this proposition at hand, we are able to prove \prettyref{thm:main}.
\begin{proof}[Proof of \prettyref{thm:main}]
 Employing \prettyref{prop:sufficient} it suffices to prove that
for an $m$-accretive operator $B\subseteq A$ we find a function
$h:\ker(1-A)\to\ker(1+A)$ with the desired properties. So, let $B\subseteq A$
be $m$-accretive. Using \prettyref{lem:inner_product_A} , we obtain
for each $u,v\in\dom(B)$
\[
0\leq\langle B\left(u\right)-B\left(v\right),u-v\rangle_{H}=\langle A(u-v),u-v\rangle_{H}=\|\pi_{1}u-\pi_{1}v\|_{H}^{2}-\|\pi_{-1}u-\pi_{-1}v\|_{H}^{2}
\]
and thus, 
\[
\forall u,v\in\dom(B):\,\|\pi_{-1}u-\pi_{-1}v\|_{H}\leq\|\pi_{1}u-\pi_{1}v\|_{H}.
\]
The latter yields, that we can define a mapping 
\begin{align*}
h\colon\{\pi_{1}u\,;\,u\in\dom(B)\}\subseteq\ker(1-A) & \to\ker(1+A),\\
\pi_{1}u & \mapsto\pi_{-1}u
\end{align*}
which is well-defined and contractive. Next, we show that $\{\pi_{1}u\,;\,u\in\dom(B)\}=\ker(1-A).$
Indeed, by \prettyref{lem:projections} we have that 
\begin{align*}
\{\pi_{1}u\,;\,u\in\dom(B)\} & =\{\tfrac{1}{2}P_{\ker(1-A)}(1+A)u\,;\,u\in\dom(B)\}\\
 & =\{\tfrac{1}{2}P_{\ker(1-A)}(1+B)\left(u\right)\,;\,u\in\dom(B)\}\\
 & =\ran P_{\ker(1-A)}=\ker(1-A),
\end{align*}
where we have used that $1+B$ is onto. Thus, $h:\ker(1-A)\to\ker(1+A)$
is a contractive mapping and by definition 
\[
\dom(B)\subseteq\{u\in\dom(A)\,;\,h(\pi_{1}u)=\pi_{-1}u\}.
\]
Denote now by $C$ the restriction of $A$ to the set $\{u\in\dom(A)\,;\,h(\pi_{1}u)=\pi_{-1}u\}$.
Then $C$ is $m$-accretive by \prettyref{prop:sufficient} and clearly,
$B\subseteq C.$ By the maximality of $B,$ we infer $B=C$ and thus,
indeed 
\[
\dom(B)=\{u\in\dom(A)\,;\,h(\pi_{1}u)=\pi_{-1}u\}.\tag*{\qedhere}
\]
\end{proof}
The mapping $h$ in \prettyref{thm:main} can be computed in terms
of the operator $B$.
\begin{prop}
Let $B\subseteq A$ be $m$-accretive and $h\colon\ker(1-A)\to\ker(1+A)$
as in \prettyref{thm:main}. Then 
\[
h(v)=\pi_{-1}((1+B)^{-1}(2v))=\frac{1}{2}P_{\ker(1+A)}\left(\left(1-A\right)(1+B)^{-1}(2v)\right)\quad(v\in\ker(1-A)).
\]
In particular, $h$ is uniquely determined. Moreover, $-A_{0}\subseteq B$
if and only if $h(0)=0$ and $B$ is linear if and only if $h$ is
linear.
\end{prop}

\begin{proof}
For $v\in\ker(1-A)$ we have that $u\coloneqq(1+B)^{-1}(2v)\in\dom(B).$
Then $v=\frac{1}{2}(1+B)\left(u\right)\in\ker(1-A)$ and hence, 
\[
v=P_{\ker(1-A)}\left(\frac{1}{2}(1+B)\left(u\right)\right)=\frac{1}{2}P_{\ker(1-A)}\left((1+A)u\right)=\pi_{1}(u),
\]
where we have used \prettyref{lem:projections}. In consequence, 
\[
h(v)=h(\pi_{1}(u))=\pi_{-1}(u)=\pi_{-1}((1+B)^{-1}(2v)).
\]
The second asserted equality is a consequence of \prettyref{lem:projections}.
Moreover, $-A_{0}\subseteq B$ if and only if $\dom(A_{0})\subseteq\dom(B)$.
Since the elements in $\dom(A_{0})$ are precisely those $u\in\dom(A)$
with $\pi_{1}u=\pi_{-1}u=0,$ we derive that $-A_{0}\subseteq B$
if and only if $h(0)=0.$ Finally, the formula for $h$ shows, that
$h$ is linear if $B$ is linear. If on the other hand $h$ is linear,
then $\dom(B)\subseteq\dom(A)$ is a subspace and hence, $B$ is linear. 
\end{proof}

\section{Two Examples}

We illustrate our findings of the previous section by two examples.
First, we consider the derivative on a bounded interval and discuss
$m$-accretive realisation of it and second, we study block-operator
matrices, which naturally arise in the study of evolutionary equations
and most (if not all) equations of mathematical physics.

\subsection{The Derivative}

Throughout, let $a,b\in\R$ with $a<b$ and consider the operator
\[
\partial_{0}:H_{0}^{1}([a,b];\R)\subseteq L_{2}([a,b];\R)\to L_{2}([a,b];\R),\quad f\mapsto f',
\]
where $H_{0}^{1}([a,b];\R)$ denotes the Sobolev-space $H^{1}([a,b];\R)$
with vanishing boundary conditions; that is,

\[
H_{0}^{1}([a,b];\R)=\{f\in L_{2}([a,b];\R)\,;\,f'\in L_{2}([a,b];\R),f(a)=f(b)=0\}.
\]
Here, $f'$ denotes the derivative of $f$ in the distributional sense
and the point-evaluations $f(a)$ and $f(b)$ are well-defined due
to Sobolev's embedding theorem. We note, that 
\[
H_{0}^{1}([a,b];\R)=\overline{C_{c}^{\infty}(\mathopen{]}a,b\mathclose{[};\R)}^{H^{1}([a,b];\R)},
\]
where $C_{c}^{\infty}(\mathopen{]a,b\mathclose{[})}$ denotes the
space of arbitrarily differential functions having compact support
in the open interval $\mathopen{]}a,b\mathclose{[}$ and the closure
is taken in the Sobolev space $H^{1}$, which is given by 
\[
H^{1}([a,b];\R)\coloneqq\{f\in L_{2}([a,b];\R)\,;\,f'\in L_{2}([a,b];\R)\}
\]
equipped with the graph inner product of the weak derivative. It is
well-known that $\partial_{0}$ is skew-symmetric and that its adjoint
is given by 
\[
\partial_{0}^{\ast}=-\partial,
\]
 where 
\[
\partial:H^{1}([a,b];\R)\subseteq L_{2}([a,b];\R)\to L_{2}([a,b];\R),\quad f\mapsto f'.
\]
We begin to compute the spaces $\ker(1\pm\partial)$ and the orthogonal
projections $\pi_{\pm1}.$ 
\begin{lem}
\label{lem:projector_derivative}We have 
\[
\ker(1\pm\partial)=\{t\mapsto c\e^{\mp t}\,;\,c\in\R\}.
\]
Moreover, 
\begin{align*}
\pi_{-1}u & =(t\mapsto\tfrac{u(a)\e^{-a}-u(b)\e^{-b}}{\e^{-2a}-\e^{-2b}}\e^{-t}),\\
\pi_{1}u & =(t\mapsto\tfrac{u(b)\e^{b}-u(a)\e^{a}}{\e^{2b}-\e^{2a}}\e^{t}).
\end{align*}
\end{lem}

\begin{proof}
The asserted equality for the kernels is clear. Let now $u\in H^{1}([a,b];\R)$.
Then $\left(\pi_{-1}u\right)(t)=c\e^{-t}$ for some $c\in\R$ and
satisfies 
\[
\langle t\mapsto\e^{-t},u-\pi_{-1}u\rangle_{H^{1}([a,b];\R)}=0.
\]
The latter gives 
\begin{align*}
\int_{a}^{b}\e^{-t}(u(t)-u'(t))\d t & =\int_{a}^{b}\e^{-t}c\e^{-t}\d t+\int_{a}^{b}\e^{-t}c\e^{-t}\d t\\
 & =2c\int_{a}^{b}\e^{-2t}\d t=c\left(\e^{-2a}-\e^{-2b}\right).
\end{align*}
The integral on the left hand side gives 
\begin{align*}
\int_{a}^{b}\e^{-t}(u(t)-u'(t))\d t & =\int_{a}^{b}\e^{-t}u(t)\d t-\left(u(b)\e^{-b}-u(a)\e^{-a}+\int_{a}^{b}\e^{-t}u(t)\d t\right)\\
 & =u(a)\e^{-a}-u(b)\e^{-b}
\end{align*}
and hence, 
\[
c=\frac{u(a)\e^{-a}-u(b)\e^{-b}}{\e^{-2a}-\e^{-2b}}.
\]
The formula for $\pi_{1}u$ follows by analogous arguments. 
\end{proof}
Since $\ker(1-\partial)$ and $\ker(1+\partial)$ are both one-dimensional,
mappings $h:\ker(1-\partial)\to\ker(1+\partial)$ are induces by mappings
$g:\R\to\R$ via 
\begin{equation}
h(t\mapsto c\e^{t})=(t\mapsto g(c)\e^{-t}).\label{eq:one_dimensional_h}
\end{equation}

\begin{lem}
\label{lem:boundary_function_derivaive}A function $h:\ker(1-\partial)\to\ker(1+\partial)$
satisfies 
\[
\forall x,y\in\ker(1-\partial):\,\|h(x)-h(y)\|_{L_{2}([a,b])}\leq\|x-y\|_{L_{2}([a,b])},
\]
if and only if 
\[
\forall c,d\in\R:\,|g(c)-g(d)|\leq\e^{a+b}|c-d|,
\]
where $h$ and $g$ are linked via \prettyref{eq:one_dimensional_h}.
\end{lem}

\begin{proof}
The proof is straight forward. For $x,y\in\ker(1-\partial)$ we find
$c,d\in\R$ such that $x=(t\mapsto c\e^{t}),y=(t\mapsto d\e^{t}).$
We then compute 
\begin{align*}
\|h(x)-h(y)\|_{L_{2}([a,b])}^{2} & =\int_{a}^{b}|g(c)-g(d)|^{2}\e^{-2x}\d x\\
 & =|g(c)-g(d)|^{2}\frac{1}{2}(\e^{-2a}-\e^{-2b})
\end{align*}
and on the other hand 
\[
\|x-y\|_{L_{2}([a,b])}^{2}=\int_{a}^{b}|c-d|^{2}\e^{2x}\d x=|c-d|^{2}\frac{1}{2}(\e^{2b}-\e^{2a})
\]
and thus, the assertion follows with the help of the equality
\[
\frac{\e^{2b}-\e^{2a}}{\e^{-2a}-\e^{-2b}}=\e^{2(a+b)}.\tag*{\qedhere}
\]
 
\end{proof}
\begin{thm}
Let $B\subseteq\partial.$ Then $B$ is $m$-accretive if and only
if there exists a function $g:\R\to\R$ with 
\[
\forall c,d\in\R:\,|g(c)-g(d)|\leq\e^{a+b}|c-d|
\]
 such that 
\[
\dom(B)=\{u\in H^{1}([a,b];\R)\,;\,\tfrac{u(a)\e^{-a}-u(b)\e^{-b}}{\e^{-2a}-\e^{-2b}}=g\left(\tfrac{u(b)\e^{b}-u(a)\e^{a}}{\e^{2b}-\e^{2a}}\right)\}.
\]
\end{thm}

\begin{proof}
The statement is a direct application of \prettyref{thm:main} in
combination with \prettyref{lem:projector_derivative} and \prettyref{lem:boundary_function_derivaive}.
\end{proof}
\begin{rem}
\begin{enumerate}[(a)]

\item We note that in the symmetric case when $a=-b$ with $b>0,$
we have that 
\[
\e^{a+b}=1
\]
and thus, $g$ has to be a contraction.

\item Note that in the linear case, we can write the boundary condition
as 
\begin{equation}
\left(\frac{\e^{-b}}{\e^{-2a}-\e^{-2b}}+g\frac{\e^{b}}{\e^{2b}-\e^{2a}}\right)u(b)=\left(\frac{\e^{-a}}{\e^{-2a}-\e^{-2b}}+g\frac{\e^{a}}{\e^{2b}-\e^{2a}}\right)u(a)\label{eq:crazy_bd}
\end{equation}
for a real number $g\in\R$ with $|g|\leq\e^{a+b}.$ Note that 
\[
\frac{\e^{-a}}{\e^{-2a}-\e^{-2b}}+g\frac{\e^{a}}{\e^{2b}-\e^{2a}}=\frac{\e^{a}}{\e^{2b}-\e^{2a}}\left(\e^{2b}+g\right)\ne0
\]
for $|g|\leq\e^{a+b}<\e^{2b}$ and thus, the boundary condition can
be written as 
\[
\left(\frac{\e^{-b}}{\e^{-2a}-\e^{-2b}}+g\frac{\e^{b}}{\e^{2b}-\e^{2a}}\right)\left(\frac{\e^{-a}}{\e^{-2a}-\e^{-2b}}+g\frac{\e^{a}}{\e^{2b}-\e^{2a}}\right)^{-1}u(b)=u(a).
\]
Note that the term on the left hand side can be rewritten as 
\[
\left(\e^{b-a}\frac{\e^{2a}+g}{\e^{2b}+g}\right)u(b)=u(a).
\]
Using that the mapping 
\[
[-\e^{a+b},\e^{a+b}]\ni g\mapsto\e^{b-a}\frac{\e^{2a}+g}{\e^{2b}+g}\in[-1,1]
\]
is bijective, we obtain that \prettyref{eq:crazy_bd} is equivalent
to 
\[
cu(b)=u(a)
\]
for some $|c|\leq1,$ which recovers the well-known characterisation
of linear boundary conditions for $m$-accretive realisations of the
derivative, see e.g. \cite[Theorem 7.2.4]{Jacob_Zwart2012} or \cite[Theorem 2.8]{PTWW2022}.
It is noteworthy that in the nonlinear situation, we cannot expect
to formulate all boundary conditions as $u(a)=f(u(b))$ for a suitable
$f$. 

\end{enumerate}
\end{rem}

\subsection{Block-Operator Matrices}

In this section we inspect the following setting. We assume that $H=H_{0}\oplus H_{1}$
for two Hilbert spaces $H_{0},H_{1}$. Moreover, let $G_{0}\colon\dom(G_{0})\subseteq H_{0}\to H_{1}$
and $D_{0}\colon\dom(D_{0})\subseteq H_{1}\to H_{0}$ be two closed,
densely defined linear operators such that 
\[
G\coloneqq-D_{0}^{\ast}\supseteq G_{0}\text{ and }D\coloneqq-G_{0}^{\ast}\supseteq D_{0}.
\]
We consider the following skew-symmetric operator on $H$
\[
A_{0}\coloneqq\left(\begin{array}{cc}
0 & D_{0}\\
G_{0} & 0
\end{array}\right)
\]
together with its negative adjoint 
\begin{equation}
A=-A_{0}^{\ast}=\left(\begin{array}{cc}
0 & D\\
G & 0
\end{array}\right).\label{eq:block_op}
\end{equation}

\begin{rem}
The prototype for those operators $G_{0}$ and $D_{0}$ are the gradient
and the divergence with vanishing boundary values. More precisely,
let $\Omega\subseteq\R^{n}$ open and set
\begin{align*}
\grad_{0}\colon H_{0}^{1}(\Omega)\subseteq L_{2}(\Omega)\to L_{2}(\Omega)^{n}, & \quad f\mapsto\left(\partial_{j}f\right)_{j\in\{1,\ldots,n\}},\\
\dive_{0}\coloneqq\dom(\dive_{0})\subseteq L_{2}(\Omega)^{n}\to L_{2}(\Omega), & \quad\Phi\mapsto\sum_{j=1}^{n}\partial_{j}\Phi_{j},
\end{align*}
where 
\[
\dom(\dive_{0})=\overline{C_{c}^{\infty}(\Omega)^{n}}^{\dom(\dive)}
\]
and $\dive\coloneqq-\grad_{0}^{\ast}$, $\grad\coloneqq-\dive_{0}^{\ast}$
are the usual distributional divergence and gradient on $L_{2}(\Omega)$.
In this situation we have 
\[
H_{0}=L_{2}(\Omega),\quad H_{1}=L_{2}(\Omega)^{n},\quad G_{0}=\grad_{0},\quad D_{0}=\dive_{0},\quad G=\grad,\quad D=\dive
\]
and the resulting block operator matrix is given by 
\[
A=\left(\begin{array}{cc}
0 & \dive\\
\grad & 0
\end{array}\right).
\]
Those operators naturally occur in first-order formulations of classical
partial differential equations in mathematical physics, in particular
for the wave and heat equation. We emphasise that the same construction
works for the rotation, yielding the block operator matrix 
\[
A=\left(\begin{array}{cc}
0 & -\curl\\
\curl & 0
\end{array}\right)
\]
and similar for the gradient and divergence on higher-order tensor
fields, allowing to treat Maxwell's equations, the equation of elasticity
and coupled problems thereof. For more details we refer to \cite{Picard2009,PMTW2020,PTW2015,PSTW2016,STW2022}.
\end{rem}

Following \cite{PTW2016}, we introduce the following spaces.
\begin{defn*}
For operators $G$ and $D$ as above we set
\begin{align*}
\mathcal{BD}(G) & \coloneqq\{u\in\dom(DG)\,;\,DGu=u\},\\
\mathcal{BD}(D) & \coloneqq\{v\in\dom(GD)\,;\,GDv=v\}.
\end{align*}
Then $\mathcal{BD}(G)=\dom(G_{0})^{\bot_{\dom(G)}}$ and $\mathcal{BD}(D)=\dom(D_{0})^{\bot_{\dom(D)}}$
and hence, both spaces are closed subspaces of $\dom(G)$ and $\dom(D)$,
respectively. Moreover, $G_{\mathcal{BD}}:\mathcal{BD}(G)\to\mathcal{BD}(D)$
and $D_{\mathcal{BD}}:\mathcal{BD}(D)\to\mathcal{BD}(G)$ with $G_{\mathcal{BD}}u=Gu$
and $D_{\mathcal{BD}}v=Dv$ are unitary with $D_{\mathcal{BD}}^{\ast}=G_{\mathcal{BD}}.$
For $u\in\dom(G)$ and $v\in\dom(D)$ we denote the orthogonal projections
on $\mathcal{BD}(G)$ and $\mathcal{BD}(D)$ by $u_{\mathcal{BD}}$
and $v_{\mathcal{BD}}$, respectively.
\end{defn*}
We begin to compute the projections onto $\ker(1\pm A).$
\begin{lem}
\label{lem:ortho_projector_block_op}Let $A$ be as in \prettyref{eq:block_op}.
For $(u,v)\in\dom(A)$ we have 
\[
\pi_{1}\left(\begin{array}{c}
u\\
v
\end{array}\right)=\frac{1}{2}\left(\begin{array}{c}
u_{\mathcal{BD}}+Dv_{\mathcal{BD}}\\
Gu_{\mathcal{BD}}+v_{\mathcal{BD}}
\end{array}\right),\quad\pi_{-1}\left(\begin{array}{c}
u\\
v
\end{array}\right)=\frac{1}{2}\left(\begin{array}{c}
u_{\mathcal{BD}}-Dv_{\mathcal{BD}}\\
-Gu_{\mathcal{BD}}+v_{\mathcal{BD}}
\end{array}\right).
\]
\end{lem}

\begin{proof}
Let $(u,v)\in\dom(A)$; that is, $u\in\dom(G)$ and $v\in\dom(D).$
Then we can decompose $u=u_{0}+u_{\mathcal{BD}}$ and $v=v_{0}+v_{\mathcal{BD}}$
where $u_{0}\in\dom(G_{0})$ and $v_{0}\in\dom(D_{0})$ and $u_{0}\bot u_{\mathcal{BD}}$
and $v_{0}\bot v_{\mathcal{BD}}$ in $\dom(G)$ and $\dom(D)$ respectively.
We recall from \prettyref{lem:projections}
\[
\pi_{1}\left(\begin{array}{c}
u\\
v
\end{array}\right)=\frac{1}{2}P_{\ker(1-A)}\left(\left(1+\left(\begin{array}{cc}
0 & D\\
G & 0
\end{array}\right)\right)\left(\begin{array}{c}
u\\
v
\end{array}\right)\right).
\]
We compute for $(x,y)\in\ker(1-A)$
\begin{align*}
\langle u+Dv,x\rangle_{H_{0}}+\langle v+Gu,y\rangle_{H_{1}} & =\langle u_{0},x\rangle_{H_{0}}+\langle u_{\mathcal{BD}},x\rangle_{H_{0}}+\langle D_{0}v_{0},x\rangle_{H_{0}}+\langle Dv_{\mathcal{BD}},x\rangle_{H_{0}}\\
 & \quad+\langle v_{0},y\rangle_{H_{1}}+\langle v_{\mathcal{BD}},y\rangle_{H_{1}}+\langle G_{0}u_{0},y\rangle_{H_{1}}+\langle Gu_{\mathcal{BD}},y\rangle_{H_{1}}\\
 & =\langle u_{0},x\rangle_{H_{0}}+\langle u_{\mathcal{BD}},x\rangle_{H_{0}}-\langle v_{0},Gx\rangle_{H_{1}}+\langle Dv_{\mathcal{BD}},x\rangle_{H_{0}}\\
 & \quad+\langle v_{0},y\rangle_{H_{1}}+\langle v_{\mathcal{BD}},y\rangle_{H_{1}}-\langle u_{0},Dy\rangle_{H_{0}}+\langle Gu_{\mathcal{BD}},y\rangle_{H_{1}}.
\end{align*}
Using $\left(\begin{array}{c}
x\\
y
\end{array}\right)=A\left(\begin{array}{c}
x\\
y
\end{array}\right)=\left(\begin{array}{c}
Dy\\
Gx
\end{array}\right),$ we derive $x=Dy$ and $y=Gx$ and in particular $x\in\mathcal{BD}(G)$
and $y\in\mathcal{BD}(D)$. The latter gives 
\begin{align*}
\langle u+Dv,x\rangle_{H_{0}}+\langle v+Gu,y\rangle_{H_{1}} & =\langle u_{\mathcal{BD}},x\rangle_{H_{0}}+\langle Dv_{\mathcal{BD}},x\rangle_{H_{0}}+\langle v_{\mathcal{BD}},y\rangle_{H_{1}}+\langle Gu_{\mathcal{BD}},y\rangle_{H_{1}}\\
 & =\langle u_{\mathcal{BD}},x\rangle_{H_{0}}+\langle Dv_{\mathcal{BD}},Dy\rangle_{H_{0}}+\langle v_{\mathcal{BD}},y\rangle_{H_{1}}+\langle Gu_{\mathcal{BD}},Gx\rangle_{H_{1}}\\
 & =\langle u_{\mathcal{BD}},x\rangle_{\mathcal{BD}(G)}+\langle v_{\mathcal{BD}},y\rangle_{\mathcal{BD}(D)}.
\end{align*}
Hence, 
\begin{align*}
 & \left\langle \left(1+\left(\begin{array}{cc}
0 & D\\
G & 0
\end{array}\right)\right)\left(\begin{array}{c}
u\\
v
\end{array}\right)-\left(\begin{array}{c}
u_{\mathcal{BD}}+D_{\mathcal{BD}}v_{\mathcal{BD}}\\
G_{\mathcal{BD}}u_{\mathcal{BD}}+v_{\mathcal{BD}}
\end{array}\right),\left(\begin{array}{c}
x\\
y
\end{array}\right)\right\rangle _{H_{0}\times H_{1}}\\
 & =\langle u_{\mathcal{BD}},x\rangle_{\mathcal{BD}(G)}+\langle v_{\mathcal{BD}},y\rangle_{\mathcal{BD}(D)}\\
 & \quad-\langle u_{\mathcal{BD}}+Dv_{\mathcal{BD}},x\rangle_{H_{0}}-\langle Gu_{\mathcal{BD}}+v_{\mathcal{BD}},y\rangle_{H_{1}}\\
 & =\langle Gu_{\mathcal{BD}},Gx\rangle_{H_{1}}+\langle Dv_{\mathcal{BD}},Dy\rangle_{H_{0}}-\langle Dv_{\mathcal{BD}},x\rangle_{H_{0}}-\langle Gu_{\mathcal{BD}},y\rangle_{H_{1}}\\
 & =0
\end{align*}
and using that $\left(\begin{array}{c}
u_{\mathcal{BD}}+Dv_{\mathcal{BD}}\\
Gu_{\mathcal{BD}}+v_{\mathcal{BD}}
\end{array}\right)\in\ker(1-A),$ we infer 
\[
\pi_{1}\left(\begin{array}{c}
u\\
v
\end{array}\right)=\frac{1}{2}P_{\ker(1-A)}\left(\left(1+\left(\begin{array}{cc}
0 & D\\
G & 0
\end{array}\right)\right)\left(\begin{array}{c}
u\\
v
\end{array}\right)\right)=\frac{1}{2}\left(\begin{array}{c}
u_{\mathcal{BD}}+Dv_{\mathcal{BD}}\\
Gu_{\mathcal{BD}}+v_{\mathcal{BD}}
\end{array}\right).
\]
The second formula follows by replacing $G$ and $D$ by $-G$ and
$-D$, respectively.
\end{proof}
\begin{cor}
\label{cor:reduce_h}Let $A$ be as in \prettyref{eq:block_op}. Then
each mapping $h:\ker(1-A)\to\ker(1+A)$ is uniquely determined by
a mapping $f:\mathcal{BD}(G)\to\mathcal{BD}(G)$ via
\[
h(u,v)=(f(u),-Gf(u))\quad((u,v)\in\ker(1-A)).
\]
Moreover, 
\[
|h|_{\mathrm{Lip}}=|f|_{\mathrm{Lip}},
\]
where the Lipschitz-seminorm of $h$ is computed in $H=H_{0}\times H_{1}$
and the Lipschitz-seminorm of $f$ is computed in $\mathcal{BD}(G).$
\end{cor}

\begin{proof}
If $h(u,v)=(h_{1}(u,v),h_{2}(u,v))$ is given, then define $f(u)\coloneqq h_{1}(u,Gu)$
for each $u\in\mathcal{BD}(G).$ Note that $f$ is well-defined, since
$(u,Gu)\in\ker(1-A)$ for each $u\in\mathcal{BD}(G)$. Moreover, 
\[
h(u,v)=h(u,Gu)=(f(u),-Gf(u))\quad((u,v)\in\ker(1-A)),
\]
since $h$ attains values in $\ker(1+A)$ and thus, the second coordinate
of $h(u,v)$ is given by $-Gh_{1}(u,v)$ according to \prettyref{lem:ortho_projector_block_op}.
Similarly, if $f$ is given, we set $h:\ker(1-A)\to\ker(1+A)$ by
$h(u,v)\coloneqq(f(u),-Gf(u))$, which is well-defined by \prettyref{lem:ortho_projector_block_op}.
Finally, we observe that for $(u,v)\in\ker(1\pm A)$ we have 
\[
\left\Vert \left(\begin{array}{c}
u\\
v
\end{array}\right)\right\Vert _{H_{0}\times H_{1}}^{2}=\|u\|_{H_{0}}^{2}+\|v\|_{H_{1}}^{2}=\|u\|_{H_{0}}^{2}+\|Gu\|_{H_{1}}^{2}=\|u\|_{\mathcal{BD}(G)}^{2},
\]
from which we derive the last assertion. 
\end{proof}
Next, we provide another characterisation of $m$-accretive relations
on a Hilbert space. 
\begin{lem}
\label{lem:relation}Let $M\subseteq H\times H$ for some Hilbert
space $H$. Then $M$ is $m$-accretive, if and only if there exists
$f:H\to H$ Lipschitz-continuous with $|f|_{\mathrm{Lip}}\leq1$ such
that 
\[
M=2(1+f)^{-1}-1.
\]
In particular 
\[
v=f(u)\Leftrightarrow(u+v,u-v)\in M\quad(u,v\in H).
\]
\end{lem}

\begin{proof}
Assume first that $M$ is $m$-accretive and set 
\[
f\coloneqq\left(\frac{1}{2}(M+1)\right)^{-1}-1.
\]
Observe that for $u,v\in H$ we have that 
\[
(u,v)\in f\Leftrightarrow(u+v,u-v)\in M.
\]
Thus, for $(u,v),(x,y)\in f$ we estimate 
\begin{align*}
0 & \leq\langle(u+v)-(x+y),(u-v)-(x-y)\rangle\\
 & =\langle(u-x)+(v-y),(u-x)-(v-y)\rangle\\
 & =\|u-x\|^{2}-\|v-y\|^{2}
\end{align*}
and hence, 
\[
\|v-y\|\leq\|u-x\|.
\]
This proves that $f$ is a Lipschitz-continuous mapping with $|f|_{\mathrm{Lip}}\leq1.$
To prove that $\dom(f)=H,$ we observe that $u\in\dom(f)$ if and
only if $(u+f(u),u-f(u))\in M.$ The later is equivalent to 
\[
(u+f(u),2u)\in1+M
\]
and since $1+M$ is onto, the assertion follows. \\
If conversely $f$ is given, we set 
\[
M\coloneqq2(1+f)^{-1}-1.
\]
Then for $u,v\in H$ we have 
\[
(u,v)\in M\Leftrightarrow f\left(\frac{1}{2}(u+v)\right)=\frac{1}{2}(u-v).
\]
Hence, for $(u,v),(x,y)\in M$ we estimate 
\begin{align*}
 & \langle u-x,v-y\rangle\\
 & =\langle\tfrac{1}{2}(u+v)+\tfrac{1}{2}(u-v)-\tfrac{1}{2}(x+y)-\tfrac{1}{2}(x-y),\tfrac{1}{2}(u+v)-\tfrac{1}{2}(u-v)-\tfrac{1}{2}(x+y)+\tfrac{1}{2}(x-y)\rangle\\
 & =\tfrac{1}{4}\Re\langle\left((u+v)-(x+y)\right)+\left((u-v)-(x-y)\right),\left((u+v)-(x+y)\right)-\left((u-v)-(x-y)\right)\rangle\\
 & =\tfrac{1}{4}\|(u+v)-(x+y)\|^{2}-\tfrac{1}{4}\|(u-v)-(x-y)\|^{2}\\
 & =\|\tfrac{1}{2}(u+v)-\tfrac{1}{2}(x+y)\|^{2}-\|f\left(\tfrac{1}{2}(u+v)\right)-f\left(\tfrac{1}{2}(x+y)\right)\|^{2}\\
 & \geq0,
\end{align*}
and hence, $M$ is accretive. For showing $m$-accretivity, let $v\in H.$
We set $u\coloneqq f(\frac{1}{2}v)+\frac{1}{2}v.$ Then $(u,v)\in1+M$
since, 
\[
f\left(\frac{1}{2}\left(u+(v-u)\right)\right)=f\left(\frac{1}{2}v\right)=u-\frac{1}{2}v=\frac{1}{2}(u-(v-u)),
\]
which shows $(u,v-u)\in M.$ 
\end{proof}
\begin{thm}
\label{thm:char_max_mon_block}Let $C\subseteq A$ with $A$ as in
\prettyref{eq:block_op} Then the following statements are equivalent:

\begin{enumerate}[(i)]

\item $C$ is $m$-accretive,

\item there exists $h:\ker(1-A)\to\ker(1+A)$ with $|h|_{\mathrm{Lip}}\leq1$
such that 
\[
\dom(C)=\{x\in\dom(A)\,;\,h(\pi_{1}x)=\pi_{-1}x\}.
\]

\item there exists $f:\mathcal{BD}(G)\to\mathcal{BD}(G)$ with $|f|_{\mathrm{Lip}}\leq1$
such that 
\[
\dom(C)=\{(u,v)\in\dom(A)\,;\,f(\frac{1}{2}(u_{\mathcal{BD}}+Dv_{\mathcal{BD}}))=\frac{1}{2}(u_{\mathcal{BD}}-Dv_{\mathcal{BD}})\}.
\]

\item there exists an $m$-accretive relation $M\subseteq\mathcal{BD}(G)\times\mathcal{BD}(G)$
such that 
\[
\dom(C)=\{(u,v)\in\dom(A)\,;\,(u_{\mathcal{BD}},Dv_{\mathcal{BD}})\in M\}.
\]

\end{enumerate}

In either case we have 
\[
f(u)=h_{1}((u,Gu))\quad(u\in\mathcal{BD}(G)),\quad M=2(1+f)^{-1}-1.
\]
\end{thm}

\begin{proof}
The equivalence of (i) and (ii) follows from \prettyref{thm:main}.
The equivalence of (ii) and (iii) follows from \prettyref{cor:reduce_h}.
For the equivalence of (iii) and (iv), we observe that 
\[
(u_{\mathcal{BD}},Dv_{\mathcal{BD}})\in M\Leftrightarrow f(\frac{1}{2}(u_{\mathcal{BD}}+Dv_{\mathcal{BD}}))=\frac{1}{2}(u_{\mathcal{BD}}-Dv_{\mathcal{BD}})
\]
and thus, the statement follows from \prettyref{lem:relation}.
\end{proof}
\begin{rem}
Note that the equivalence of (i) and (iv) in the latter theorem is
the main result of \cite{Trostorff2014}.
\end{rem}

We conclude this section with the study of linear $m$-accretive relations,
which by \prettyref{thm:char_max_mon_block} correspond to linear
$m$-accretive restrictions of the operator $A$ given in \prettyref{eq:block_op}.
We start with the following simple observation.
\begin{lem}
Let $X,Y$ be Hilbert spaces and $M\subseteq X\times Y$ be a closed
subspace. Then there exists $S\in L(X;X\times Y)$ and $T\in L(Y;X\times Y)$
such that 
\[
(u,v)\in M\Leftrightarrow Su=Tv\quad(u\in X,v\in Y).
\]
\end{lem}

\begin{proof}
Since $M$ is a closed subspace of the Hilbert space $X\times Y$,
we have $(u,v)\in M$ if and only if $P_{M^{\bot}}(u,v)=0$ for each
$u\in X,v\in Y.$ Hence, we may set $S\coloneqq P_{M^{\bot}}\iota_{X}$
and $T\coloneqq-P_{M^{\bot}}\iota_{Y},$ where $\iota_{X}$ and $\iota_{Y}$
denote the canonical embeddings of $X$ and $Y$ in $X\times Y$,
respectively. Then 
\[
Su=Tv\Leftrightarrow P_{M^{\bot}}(\iota_{x}u+\iota_{Y}v)=0\Leftrightarrow(u,v)\in M.\tag*{\qedhere}
\]
\end{proof}
Since linear $m$-accretive relations $M$ on a Hilbert space $X$
are closed (see e.g. \prettyref{prop:accretive_dense_range}), $m$-accretive
relations can be described by operator equalities as in the lemma
above. The natural question which arises is: when is a relation, which
is determined by two operators as above, $m$-accretive?
\begin{prop}
Let $X$ be a Hilbert space and $Y$ a normed space. Let $S,T\in L(X,Y)$
and consider the relation 
\[
M\coloneqq\{(u,v)\in X\times X\,;\,Su=Tv\}.
\]
Then $M$ is $m$-accretive, if and only if the following three properties
are satisfied
\begin{itemize}
\item $\ran(T-S)\subseteq\ran(T+S),$
\item $T+S$ is one-to-one,
\item $\|(S+T)^{-1}(T-S)\|\leq1.$
\end{itemize}
\end{prop}

\begin{proof}
By \prettyref{lem:relation} we know that $M$ is maximal monotone
if and only if $f\coloneqq\left(\frac{1}{2}(M+1)\right)^{-1}-1$ defines
a Lipschitz-continuous mapping on $X$ with $|f|_{\mathrm{Lip}}\leq1.$
Since $M$ is linear, we need to show that $f\in L(X)$ with $\|f\|\leq1.$
We recall from \prettyref{lem:relation} that 
\[
(u,v)\in f\Leftrightarrow(u+v,u-v)\in M\Leftrightarrow S(u+v)=T(u-v)\Leftrightarrow(S+T)v=(T-S)u.
\]
Then $\dom(f)=X$ is equivalent to $\ran(T-S)\subseteq\ran(S+T)$.
Moreover, $f$ is a mapping if and only if $(0,v)\in f$ implies $v=0.$
The latter is equivalent to the injectivity of $S+T.$ Finally, the
condition $\|f\|\leq1$ is equivalent to $\|(S+T)^{-1}(T-S)\|\leq1,$
which shows the assertion.
\end{proof}
If we apply the latter proposition to linear restrictions of the block
operator $A$ in \prettyref{eq:block_op} we obtain the following
characterisation.
\begin{cor}
\label{cor:impedance_coefficients}Let $C\subseteq A$ be a linear
restriction of the block operator $A$ given in \prettyref{eq:block_op}.
Then the following statements are equivalent:

\begin{enumerate}[(i)]

\item $C$ is $m$-accretive,

\item There exists a normed space $Y$ and operators $S,T\in L(\mathcal{BD}(G),Y)$
with
\begin{itemize}
\item $\ran(T-S)\subseteq\ran(T+S)$,
\item $T+S$ is one-to-one,
\item $\|(S+T)^{-1}(T-S)\|\leq1$,
\end{itemize}
such that 
\[
\dom(C)=\{(u,v)\in\dom(A)\,;\,Su_{\mathcal{BD}}=TDv_{\mathcal{BD}}\}.
\]

\end{enumerate}
\end{cor}

\section{Application to the Wave Equation with Impedance Boundary Conditions }

We consider the wave equation on a open set $\Omega\subseteq\R^{n}$
with impedance type conditions. Employing the framework of evolutionary
equations, we may write the equation as a system of the form 
\[
\left(\partial_{t}M(\partial_{t})+A\right)U=F,
\]
where $\partial_{t}$ stands for the temporal derivative and $M$
is a material law operator, incorporating all physical coefficients.
The operator $A$ is a suitable restriction of the block operator
matrix 
\[
\left(\begin{array}{cc}
0 & \dive\\
\grad & 0
\end{array}\right),
\]
where $\grad$ and $\dive$ are defined as in the previous section,
that is, 
\begin{align*}
\grad & :H^{1}(\Omega)\subseteq L_{2}(\Omega)\to L_{2}(\Omega)^{n},\quad f\mapsto(\partial_{j}f)_{j\in\{1,\ldots,n\}},\\
\dive & :H(\dive,\Omega)\subseteq L_{2}(\Omega)^{n}\to L_{2}(\Omega),\quad\Phi\mapsto\sum_{i=1}^{n}\partial_{i}\Phi_{i},
\end{align*}
where 
\[
H(\dive,\Omega)\coloneqq\{\Phi\in L_{2}(\Omega)^{n}\,;\,\sum_{i=1}^{n}\partial_{i}\Phi_{i}\in L_{2}(\Omega)\}
\]
and the derivatives are computed in the distributional sense. We recall
that in case of bounded set $\Omega$ with Lipschitz-boundary $\Gamma,$
we may define the trace operator 
\begin{align*}
\gamma_{0}\colon H^{1}(\Omega) & \to H^{1/2}(\Gamma),\quad f\mapsto f|_{\Gamma},
\end{align*}
which is a surjective bounded operator, see e.g. \cite[Chapter 1, Theorem 1.2]{Necas2012}.
Moreover, $\ker(\gamma_{0})=H_{0}^{1}(\Omega)$ and so, we may restrict
$\gamma_{0}$ to the space $\mathcal{BD}(\grad)\coloneqq H_{0}^{1}(\Omega)^{\bot_{H^{1}(\Omega)}}$
and obtain an isomorphism between the spaces $\mathcal{BD}(\grad)$
and $H^{1/2}(\Gamma).$ Hence, we can equip $H^{1/2}(\Gamma)$ with
the equivalent norm
\[
\|\gamma_{0}f\|_{H^{1/2}(\Gamma)}\coloneqq\|f\|_{H^{1}(\Omega)}\quad(f\in\mathcal{BD}(\grad))
\]
such that $\gamma_{0}\colon\mathcal{BD}(\grad)\to H^{1/2}(\Gamma)$
becomes unitary. In a similar way, we may define a trace operator
on $H(\dive,\Omega)$ by setting 
\[
\gamma_{\mathrm{n}}\colon H(\dive,\Omega)\to H^{-1/2}(\Gamma),
\]
where $H^{-1/2}(\Gamma)$ denotes the dual space of $H^{1/2}(\Gamma)$
and 
\begin{equation}
\langle\gamma_{\mathrm{n}}\Phi,\gamma_{0}f\rangle\coloneqq\langle\Phi,\grad f\rangle_{L_{2}(\Omega)^{n}}+\langle\dive\Phi,f\rangle_{L_{2}(\Omega)}\quad(\Phi\in H(\dive,\Omega),f\in H^{1}(\Omega)).\label{eq:Green}
\end{equation}
For $f\in C^{1}(\Omega)\cap H^{1}(\Omega)$ and $\Phi\in C^{1}(\Omega)^{n}\cap H(\dive,\Omega)$
one obtains $\gamma_{0}f=f|_{\Gamma}$ and $\gamma_{\mathrm{n}}\Phi=\Phi|_{\Gamma}\cdot n,$
where $n$ denotes the unit outward directed normal on $\Gamma.$
The definition of $\gamma_{\mathrm{n}}$ gives 
\begin{align*}
\ker(\gamma_{\mathrm{n}}) & =\{\Phi\in H(\dive,\Omega)\,;\,\forall f\in H^{1}(\Omega):\,\langle\Phi,\grad f\rangle=-\langle\dive\Phi,f\rangle_{L_{2}(\Omega)}\}\\
 & =\dom(\dive_{0}),
\end{align*}
where $\dive_{0}$ is the closure of of $\dive|_{C_{c}^{\infty}(\Omega)^{n}}.$
Hence, we may restrict $\gamma_{\mathrm{n}}$ to $\dom(\dive_{0})^{\bot_{H(\dive,\Omega)}}=\mathcal{BD}(\dive)$
and obtain an isomorphism from $\mathcal{BD}(\dive)$ to $H^{-1/2}(\Gamma).$
Then \prettyref{eq:Green} reads as 
\begin{align*}
\langle\gamma_{\mathrm{n}}\Phi,\gamma_{0}f\rangle & =\langle\Phi,\grad f\rangle_{L_{2}(\Omega)^{n}}+\langle\dive\Phi,f\rangle_{L_{2}(\Omega)}\\
 & =\langle\Phi,\grad f\rangle_{L_{2}(\Omega)^{n}}+\langle\dive\Phi,\dive\grad f\rangle_{L_{2}(\Omega)}\\
 & =\langle\Phi,\grad_{\mathcal{BD}}f\rangle_{\mathcal{BD}(\dive)}\\
 & =\langle\dive_{\mathcal{BD}}\Phi,f\rangle_{\mathcal{BD}(\grad)}\quad(f\in\mathcal{BD}(\grad),\Phi\in\mathcal{BD}(\dive)).
\end{align*}
In particular, keeping in mind that we have renormed $H^{1/2}(\Gamma)$,
we have 
\[
\|\gamma_{\mathrm{n}}\Phi\|=\|\dive_{\mathcal{BD}}\Phi\|_{\mathcal{BD}(\grad)}=\|\Phi\|_{\mathcal{BD}(\dive)}\quad(\Phi\in\mathcal{BD}(\dive))
\]
and hence, $\gamma_{\mathrm{n}}$ is unitary from $\mathcal{BD}(\dive)$
to $H^{-1/2}(\Gamma)$. In order to formulate boundary conditions
of impedance type for the operator $\left(\begin{array}{cc}
0 & \dive\\
\grad & 0
\end{array}\right)$ one needs to compare the traces $\gamma_{0}f$ and $\gamma_{\mathrm{n}}\Phi$
for $f\in\mathcal{BD}(\grad)$ and $\Phi\in\mathcal{BD}(\dive).$
We consider the following type of boundary conditions 
\[
Sg=Th,\quad(g\in H^{1/2}(\Gamma),h\in H^{-1/2}(\Gamma))
\]
where $S\in L(H^{1/2}(\Gamma),Y)$ and $T\in L(H^{-1/2}(\Gamma),Y)$
for some normed space $Y$. Then a direct consequence of \prettyref{cor:impedance_coefficients}
is the following result.
\begin{prop}
Let $\Omega$ be a bounded Lipschitz domain. Let $S\in L(H^{1/2}(\Gamma),X)$
and $T\in L(H^{-1/2}(\Gamma),X)$ for some normed space $X$. Consider
$C\subseteq\left(\begin{array}{cc}
0 & \dive\\
\grad & 0
\end{array}\right)$ with 
\[
\dom(C)\coloneqq\left\{ (f,\Phi)\in H^{1}(\Omega)\times H(\dive,\Omega)\,;\,S\gamma_{0}f=T\gamma_{\mathrm{n}}\Phi\right\} .
\]
Then $C$ is $m$-accretive if and only if 
\begin{itemize}
\item $S\gamma_{0}+T\gamma_{\mathrm{n}}\grad_{\mathcal{BD}}$ is one-to-one
on $\mathcal{BD}(\grad),$
\item $\ran(T\gamma_{\mathrm{n}}\grad_{\mathcal{BD}}-S\gamma_{0})\subseteq\ran(T\gamma_{\mathrm{n}}\grad_{\mathcal{BD}}+S\gamma_{0}),$
\item $\|(T\gamma_{\mathrm{n}}\grad_{\mathcal{BD}}+S\gamma_{0})^{-1}(T\gamma_{\mathrm{n}}\grad_{\mathcal{BD}}-S\gamma_{0})\|\leq1.$
\end{itemize}
\end{prop}

\begin{proof}
Note that since $\ker(\gamma_{0})=H_{0}^{1}(\Omega)$ and $\ker(\gamma_{\mathrm{n}})=\dom(\dive_{0})$
we have 
\begin{align*}
\dom(C) & \coloneqq\left\{ (f,\Phi)\in H^{1}(\Omega)\times H(\dive,\Omega)\,;\,S\gamma_{0}f_{\mathcal{BD}}=T\gamma_{\mathrm{n}}\Phi_{\mathcal{BD}}\right\} \\
 & =\left\{ (f,\Phi)\in H^{1}(\Omega)\times H(\dive,\Omega)\,;\,S\gamma_{0}f_{\mathcal{BD}}=T\gamma_{\mathrm{n}}\dive_{\mathcal{BD}}\grad_{\mathcal{BD}}\Phi_{\mathcal{BD}}\right\} 
\end{align*}
and now the assertion follows from \prettyref{cor:impedance_coefficients}
applied to $G=\grad$ and $D=\dive.$
\end{proof}
Typically, the trace operators are compared via the pivot space $L_{2}(\Gamma).$
Indeed, one can show that the embedding
\[
\iota\colon H^{1/2}(\Gamma)\to L_{2}(\Gamma),\quad g\mapsto g
\]
is continuous, injective and has dense range, see \cite[Chapter 2, Theorem 4.9]{Necas2012}.
Consequently, the dual operator 
\[
\iota'\colon L_{2}(\Gamma)\to H^{-1/2}(\Gamma)
\]
is also continuous injective and has dense range (here we identified
$L_{2}(\Gamma)'$ with $L_{2}(\Gamma)$ by the usual Riesz isomorphism).
Impedance type boundary conditions then typically take the form 
\[
K\iota\gamma_{0}f=\gamma_{\mathrm{n}}\Phi,
\]
for an operator $K\in L(L_{2}(\Gamma))$. 
\begin{lem}
\label{lem:char_pivot}Let $\Omega$ be a bounded Lipschitz domain.
Let $f\in H^{1}(\Omega),$ and $\Phi\in H(\dive,\Omega)$ and $K\in L(L_{2}(\Gamma)).$
Then $K\iota\gamma_{0}f=\gamma_{\mathrm{n}}\Phi$ if and only if 
\[
\dive_{\mathcal{BD}}\Phi_{\mathcal{BD}}=\kappa^{\ast}K\kappa f_{\mathcal{BD}},
\]
where $\kappa\coloneqq\iota\gamma_{0}\colon\mathcal{BD}(\grad)\to L_{2}(\Gamma).$
\end{lem}

\begin{proof}
Assume first that $K\gamma_{0}f=\gamma_{\mathrm{n}}\Phi;$ that is,
for each $g\in H^{1/2}(\Gamma)$ we have 
\[
\langle\gamma_{n}\Phi,g\rangle=\langle K\iota\gamma_{0}f,\iota g\rangle_{L_{2}(\Gamma)}.
\]
The we compute for each $v\in\mathcal{BD}(\grad)$
\begin{align*}
\langle\dive_{\mathcal{BD}}\Phi_{\mathcal{BD}},v\rangle_{\mathcal{BD}(\grad)} & =\langle\gamma_{\mathrm{n}}\Phi,\gamma_{0}v\rangle\\
 & =\langle K\iota\gamma_{0}f,\iota\gamma_{0}v\rangle_{L_{2}(\Gamma)}\\
 & =\langle\kappa^{\ast}K\kappa f_{\mathcal{BD}},v\rangle_{\mathcal{BD}(\grad)},
\end{align*}
where we have used $\gamma_{0}f=\gamma_{0}f_{\mathcal{BD}}$ and $\gamma_{\mathrm{n}}\Phi=\gamma_{\mathrm{n}}\Phi_{\mathcal{BD}}$.
The latter gives $\dive_{\mathcal{BD}}\Phi_{\mathcal{BD}}=\kappa^{\ast}K\kappa f_{\mathcal{BD}}.$
Assume now conversely, that $\dive_{\mathcal{BD}}\Phi_{\mathcal{BD}}=\kappa^{\ast}K\kappa f_{\mathcal{BD}}$
and let $g\in H^{1/2}(\Gamma).$ Then we find $v\in\mathcal{BD}(\grad)$
such that $g=\gamma_{0}v.$ Hence, 
\begin{align*}
\langle\gamma_{\mathrm{n}}\Phi,g\rangle & =\langle\gamma_{\mathrm{n}}\Phi,\gamma_{0}v\rangle\\
 & =\langle\dive_{\mathcal{BD}}\Phi_{\mathcal{BD}},v\rangle_{\mathcal{BD}(\grad)}\\
 & =\langle\kappa^{\ast}K\kappa f_{\mathcal{BD}},v\rangle_{\mathcal{BD}(\grad)}\\
 & =\langle K\kappa f_{\mathcal{BD}},\kappa v\rangle_{L_{2}(\Gamma)}\\
 & =\langle K\iota\gamma_{0}f,\iota g\rangle_{L_{2}(\Gamma)}
\end{align*}
 and hence, $\gamma_{\mathrm{n}}\Phi=K\iota\gamma_{0}f.$ 
\end{proof}
\begin{prop}
Let $\Omega$ be a bounded Lipschitz domain. Let $K\in L(L_{2}(\Gamma))$.
Consider $C\subseteq\left(\begin{array}{cc}
0 & \dive\\
\grad & 0
\end{array}\right)$ with 
\[
\dom(C)\coloneqq\left\{ (f,\Phi)\in H^{1}(\Omega)\times H(\dive,\Omega)\,;\,K\iota\gamma_{0}f=\gamma_{\mathrm{n}}\Phi\right\} .
\]
Then $C$ is $m$-accretive if and only if $K$ is accretive.
\end{prop}

\begin{proof}
By \prettyref{lem:char_pivot} we have 
\[
(f,\Phi)\in\dom(C)\Leftrightarrow\dive_{\mathcal{BD}}\Phi_{\mathcal{BD}}=\kappa^{\ast}K\kappa f_{\mathcal{BD}}.
\]
Moreover, by \prettyref{thm:char_max_mon_block} (iv) $C$ is $m$-accretive
if and only if $\kappa^{\ast}K\kappa$ is $m$-accretive, which is
equivalent to the accretivity of $\kappa^{\ast}K\kappa,$ since it
is a bounded operator. Finally, $\kappa^{\ast}K\kappa$ is accretive,
if and only if $K$ is accretive, since $\kappa$ has a dense range. 
\end{proof}
\begin{rem}
As the previous example shows, classical trace spaces can be incorporated
within the framework of the corresponding $\mathcal{BD}$-spaces.
Note however, that $\mathcal{BD}$-spaces can also be used, when no
traces are available, for instance if the boundary of the underlying
domain is not smooth enough. Thus, $\mathcal{BD}$-spaces provide
a unified framework to formulate abstract boundary conditions without
any restrictions on the underlying domain. 
\end{rem}

\end{document}